\documentclass[12pt]{amsart}
\usepackage{amssymb,amsmath,amsfonts,latexsym,amscd}
\usepackage{bm}
\newcommand{\newsection}[1]{\setcounter{equation}{0} \section{#1}}
\theoremstyle{plain}
\newtheorem{propn}{Proposition}[section]
\newtheorem{thm}[propn]{Theorem}

\newtheorem{cor}[propn]{Corollary}
\newtheorem*{thm*}{Theorem}

\theoremstyle{definition}

\newtheorem*{note}{Note}

\theoremstyle{remark}
\newtheorem*{rem}{Remark}

\newcommand{\q}{\mathcal{Q}}
\newcommand{\s}{\mathcal{S}}

\newcommand{\cle}{\mathcal{E}}

\newcommand{\clh}{\mathcal{H}}

\newcommand{\clm}{\s}

\newcommand{\clq}{\mathcal{Q}}

\newcommand{\cls}{\mathcal{S}}

 \newcommand{\D}{\mathbb{D}}
\newcommand{\vp}{\varphi}

\newtheorem{Theorem}{\sc Theorem}[section]
\newtheorem{Lemma}[Theorem]{\sc Lemma}
\newtheorem{Proposition}[Theorem]{\sc Proposition}
\newtheorem{Corollary}[Theorem]{\sc Corollary}
\newtheorem{Definition}[Theorem]{\sc Definition}
\newtheorem{Example}[Theorem]{\sc Example}
\newtheorem{Remark}[Theorem]{\sc Remark}
\newtheorem{Note}[Theorem]{\sc Note}
\newtheorem{Question}{\sc Question}
\newtheorem{ass}[Theorem]{\sc Assumption}
\newcommand{\bt}{\begin{Theorem}}
\def\beginlem{\begin{Lemma}}
\def\beginprop{\begin{Proposition}}
\def\begincor{\begin{Corollary}}
\def\begindef{\begin{Definition}}
\def\beginexamp{\begin{Example}}
\def\beginrem{\begin{Remark}}
\def\beginq{\begin{Question}}
\def\beginass{\begin{ass}}
\def\beginnote{\begin{Note}}
\newcommand{\et}{\end{Theorem}}
\def\endlem{\end{Lemma}}
\def\endprop{\end{Proposition}}
\def\endcor{\end{Corollary}}
\def\enddef{\end{Definition}}
\def\endexamp{\end{Example}}
\def\endrem{\end{Remark}}
\def\endq{\end{Question}}
\def\endass{\end{ass}}
\def\endnote{\end{Note}}

\begin{document}

\title{Rank of a co-doubly commuting submodule is $2$}

\dedicatory{Dedicated to the memory of our friend and colleague
Sudipta Dutta}

\author[Chattopadhyay] {Arup Chattopadhyay}
\address{%\vskip10pt
(A. Chattopadhyay) Indian Institute of Technology Guwahati \\
Department of Mathematics \\ Amingaon Post  \\ Guwahati \\ 781039
\\ Assam, India}

\email{arupchatt@iitg.ernet.in, 2003arupchattopadhyay@gmail.com}

\author[Das] {B. Krishna Das}

\address{%\vskip10pt
(B. K. Das) Indian Institute of Technology Bombay\\
Department of Mathematics \\ Powai\\ Mumbai \\ 400076 \\
India} \email{dasb@math.iitb.ac.in, bata436@gmail.com}

\author[Sarkar]{Jaydeb Sarkar}

\address{%\vskip10pt
(J. Sarkar) Indian Statistical Institute \\ Statistics and
Mathematics Unit \\ 8th Mile, Mysore Road \\ Bangalore \\ 560059
\\ India}

\email{jay@isibang.ac.in, jaydeb@gmail.com}


%\today

\begin{abstract}
We prove that the rank of a non-trivial co-doubly commuting
submodule is $2$. More precisely, let $\varphi, \psi \in
H^\infty(\mathbb{D})$ be two inner functions. If
$\mathcal{Q}_{\varphi} = H^2(\mathbb{D})/ \varphi H^2(\mathbb{D})$
and $\mathcal{Q}_{\psi} = H^2(\mathbb{D})/ \psi H^2(\mathbb{D})$,
then
\[
\mbox{rank~}(\mathcal{Q}_{\varphi} \otimes \mathcal{Q}_{\psi})^\perp
= 2.
\]
An immediate consequence is the following: Let $\mathcal{S}$ be a
co-doubly commuting submodule of $H^2(\mathbb{D}^2)$. Then
$\mbox{rank~} \mathcal{S} = 1$ if and only if $\mathcal{S} = \Phi
H^2(\D^2)$ for some one variable inner function $\Phi \in
H^\infty(\D^2)$. This answers a question posed by R. G. Douglas and
R. Yang \cite{DY}.
\end{abstract}

\maketitle

\newsection{Introduction}\label{sec:1}

Let $T = (T_1, \ldots, T_n)$ be an $n$-tuple of commuting bounded
linear operators on a Hilbert space $\clh$. For a subset $E
\subseteq \clh$ we denote $[E]_T$ by the close subspace
$\overline{\mbox{span}}\{ T_1^{k_1} \cdots T_n^{k_n} E: k_j \in
\mathbb{N}, j = 1, \ldots, n \}$ of $\clh$. Then the \textit{rank} of $T$
\cite{DP} is the unique number
\[
\mbox{rank}(T) = \min \{\# E : [E]_T = \clh, E \subseteq \clh\}.
\]
A closed
subspace $\cls$ of $H^2(\mathbb{D}^n)$, the Hardy
space over the unit polydisc $\mathbb{D}^n$, is said to be shift invariant if
$M_{z_i}(\cls)\subseteq \cls$ for $i=1,2,\ldots,n$, where $M_{z_i}$ is the co-ordinate multiplication operator on $H^2(\mathbb{D}^n)$.
The \textit{rank} of a shift
invariant subspace $\cls$ of $H^2(\mathbb{D}^n)$ is the rank of the corresponding $n$-tuple of restricted co-ordinate shift operators, that is
\[
\mbox{rank~} \clm = \mbox{rank~}(M_{z_1}|_{\clm},\ldots, M_{z_n}|_{\clm}).
\]

The rank of a bounded linear operator (or, of a commuting tuple of
bounded linear operators) on a Hilbert space is an important
numerical invariant. Very briefly, the rank of a bounded linear
operator is the cardinality of a minimal generating set (see the
definition below). One of the most intriguing and important problems
in operator theory and function theory is the existence of a
finite generating set for a commuting tuple of operators. Alternatively, one may ask
when the rank of a commuting tuple of operators is finite.

Prototype examples of rank one operators are the co-ordinate multiplication
operator tuple $(M_{z_1}, \ldots, M_{z_n})$ on the Hardy space, the
(weighted) Bergman space over the unit ball and the polydisc in
$\mathbb{C}^n$, $n \geq 1$, and the Drury-Arveson space over the
unit ball in $\mathbb{C}^n$. Moreover, a particular version of the
celebrated invariant subspace theorem of Beurling says: A shift
invariant (or, shift co-invariant) subspace of the one variable
Hardy space is of rank one.

Computation of ranks of shift invariant as well as shift
co-invariant subspaces beyond the case of the one variable Hardy
space is an excruciatingly difficult problem, even if one considers
only shift invariant (as well as co-invariant) subspaces of the
Hardy space over the unit polydisc in $\mathbb{C}^n$, $n > 1$ (see
however \cite{CDS2, III, I31, I32, Y}).

The purpose of this paper is to compute the rank of a tractable
class of shift invariant subspaces of the two variables Hardy space,
$H^2(\D^2)$, over the bidisc $\D^2$ in $\mathbb{C}^2$. In order to
state the precise contribution of this paper, we need to introduce
first some definitions and notations.

%The purpose of this paper is to compute the rank of a simple class
%of submodules of the Hardy module over bidisc. To be more precise,
%we first recall some definitions and notations.

We denote the open unit disc of $\mathbb{C}$ by $\D$, and the unit
circle by $\mathbb{T}$. The Hardy space over the unit disc $\D$
(bidisc $\D^2$), denoted by $H^2(\D)$ ($H^2(\D^2)$), is the Hilbert
space of all square summable holomorphic functions on $\D$ (on
$\D^2$). Also we will denote by $M_z$ and $M_w$ the multiplication
operators on $H^2(\D^2)$ by the coordinate functions $z$ and $w$,
respectively. It is easy to see that $(M_{z}, M_{w})$ is a pair of
commuting isometries, that is,
\[
M_{z} M_{w} = M_{w} M_{z}, \quad M_{z}^* M_{z} = M_{w}^* M_{w} =
I_{H^2(\D^2)}.
\]
Identifying $H^2(\D^2)$ with the $2$-fold Hilbert space tensor
product $H^2(\D) \otimes H^2(\D)$, one can represent $(M_{z},
M_{w})$ as $(M_{z} \otimes I_{H^2(\D)}, I_{H^2(\D)}\otimes M_{w})$.

Let $\cls$ and $\clq$ be closed subspaces of $
H^2(\mathbb{D}^2)$. Then $\cls$ is said to be a \textit{submodule}
if $M_{z}(\mathcal{S})\subseteq \mathcal{S}$ and
$M_{w}(\mathcal{S})\subseteq \mathcal{S}$. We say that $\clq$ is a
\textit{quotient module} if $\mathcal{Q}^{\perp}$ is a submodule.

%Let $\cls$ be a submodule of $H^2(\D^2)$. Then for simplicity of
%notation, let us set
%\[
%\mbox{rank~} \clm = \mbox{rank~}(M_z|_{\clm}, M_w|_{\clm}).
%\]

A well-known result due to Beurling states that if $\cls$ is a
submodule of $H^2(\mathbb{D})$ (that is, $\cls$ is closed subspace
of $H^2(\mathbb{D})$ and $M_z \cls \subseteq \cls$), then $\cls$ can
be represented as
\[
\cls = \cls_{\varphi} : = \varphi H^2(\mathbb{D}),
\]
where $\varphi \in H^\infty(\mathbb{D})$ is an inner function (that
is, $\varphi$ is a bounded holomorphic function on $\D$ and
$|\varphi| = 1$ a.e. on $\mathbb{T}$). Consequently, a quotient
module $\clq$ (that is, $\clq$ is a closed subspace of $H^2(\D)$ and
$M_z^* \clq \subseteq \clq$) of $H^2(\mathbb{D})$ can be represented
as
\[
\clq = \clq_{\varphi} := (\cls_{\varphi})^\perp = H^2(\mathbb{D})/
\varphi H^2(\mathbb{D}).
\]
It readily follows that
\[
\mbox{rank~} (M_z|_{\cls_\varphi}) = \mbox{rank~}(P_{\clq_{\varphi}}
M_z|_{\clq_{\varphi}}) = 1.
\]
Rudin \cite{R}, however, pointed out that there exists a submodule
$\cls$ of $H^2(\D^2)$ such that the rank of $\cls$ is not finite
(see also \cite{I31}, \cite{Se} and \cite{SY}).

A quotient module $\q$ of $H^2(\D^2)$ is \emph{doubly commuting} if
$C_{z}C_{w}^*=C_{w}^*C_{z}$, where $C_{z} = P_{\q}M_{z}|_{\q}$ and
$C_{w} = P_{\q}M_{w}|_{\q}$. A submodule $\s$ of $H^2(\D^2)$ is
\emph{co-doubly commuting} if the quotient module $\s^{\perp} (\cong
H^2(\D^2)/\cls)$ is doubly commuting.

The following useful characterization of co-doubly commuting
submodules is essential for our study (see \cite{bshift, sarkar}):
If $\clq$ is a quotient module of $H^2(\D^2)$, then $\clq$ is a
doubly commuting quotient module if and only if
\[
\clq = \clq_1 \otimes \clq_2,
\]
for some quotient modules $\clq_1$ and $\clq_2$ of $H^2(\D)$.

Let $\clm = (\clq_1 \otimes \clq_2)^\perp$ be a non-zero co-doubly
commuting submodule. If $\clq_j = H^2(\D)$, for some $j = 1, 2$,
then it is easy to see that
\[
\mbox{rank~} \cls = 1.
\]
Now let both $\clq_1$ and $\clq_2$ be non-trivial quotient modules
of $H^2(\D)$, that is, $\clq_j \neq \{0\}, H^2(\D)$, $j = 1, 2$.
Then there exist inner functions $\varphi, \psi \in H^\infty(\D)$
such that $\clq_1 = \clq_{\varphi}$ and $\clq_2 = \clq_{\psi}$. The
main purpose of the present paper is to prove that (see Theorem
\ref{mainthm})
\[
\mbox{rank~} (\q_{\varphi}\otimes \q_{\psi})^{\perp} = 2.
\]
As a consequence of this, we give a complete and affirmative answer
to a conjecture of Douglas and Yang (see page 220 \cite{DY}): If
$\clm$ is a rank one co-doubly commuting submodule, then $\clm =
\Phi H^2(\D^2)$ for some one variable inner function $\Phi \in
H^\infty(\D)$.

\newsection{Proof of the main result}

We begin with a simple but crucial observation on the rank of a joint
semi-invariant subspace of a commuting tuple of operators.

\begin{Lemma}\label{Lemma}
Let $T = (T_1, \ldots, T_n)$ be an $n$-tuple of commuting operators
on a Hilbert space $\clh$. Let $\cls_1$ and $\cls_2$ be two joint
$T$-invariant subspaces of $\clh$ and $\cls_2 \subseteq \cls_1$. If
$\cls = \cls_1 \ominus \cls_2$, then
\[
\mbox{rank~} (P_{\cls} T_1|_{\cls}, \ldots, P_{\cls} T_n|_{\cls})
\leq \mbox{rank~} (T_1|_{\cls_1}, \ldots, T_n|_{\cls_1}).
\]
\end{Lemma}
\begin{proof}
Let $m \in \mathbb{N}$ be the right side of the above inequality.
Let $\{f_j\}_{j=1}^m \subseteq \cls_1$ be a generating set for
$(T_1|_{\cls_1}, \ldots, T_n|_{\cls_1})$. Clearly, $P_{\cls} T_j
P_{\cls} = P_{\cls} T_j|_{\cls_1}$ for all $j = 1, \ldots, n$. This
yields
\[
(P_{\cls} T_i P_{\cls}) (P_{\cls} T_j P_{\cls}) = P_{\cls} (T_i
T_j)|_{\cls_1} \quad \quad (i, j = 1, \ldots, n).
\]
It hence follows that $\{P_{\cls} f_j\}_{j=1}^m$ is a generating set
for $(P_{\cls} T_1|_{\cls}, \ldots, P_{\cls} T_n|_{\cls})$. This
completes the proof.
\end{proof}

We now prove the main result of this paper.

\begin{thm}\label{mainthm}
Let $\varphi, {\psi}\in H^{\infty}(\D)$ be two inner functions. If
\[
\cls = \left(\q_{\varphi}\otimes \q_{\psi}\right)^{\perp},
\]
then $\mbox{rank~} \cls = 2$.
\end{thm}
\begin{proof}
Let $X = I_{H^2(\D^2)} - (I_{H^2(\D^2)} - M_{\varphi} M_{\varphi}^*
\otimes I_{H^2(\D)})(I_{H^2(\D^2)} - I_{H^2(\D)} \otimes M_{\psi}
M_{\psi}^*)$. Since
\[
\cls = \mbox{ran} X,
\]
and
\[
X = ((M_{\varphi} M_{\varphi}^*) \otimes (I_{H^2(\D)} - M_{\psi} M_{\psi}^*))
\oplus (I_{H^2(\D)} \otimes M_{\psi} M_{\psi}^*),
\]
it follows that
\[
\cls = \left( \s_{\varphi} \otimes \q_{{\psi}}\right) \oplus
\left(H^2(\D) \otimes \s_{{\psi}} \right).
\]
Since by Theorem 6.2 of \cite{CDS}, $\mbox{rank~} \cls \leq 2$, we
only need to show that $\mbox{rank~} \cls \geq 2$. Set
\[
\mathcal{E} = \cls \ominus \left( \s_{\varphi} \otimes
\s_{{\psi}}\right).
\]
It follows that
\[
\mathcal{E} = \left( \s_{\varphi} \otimes \q_{{\psi}}\right) \oplus
\left( \q_{\varphi} \otimes \s_{{\psi}}\right).
\]
Since $\s_{\varphi} \otimes \s_{{\psi}} \subseteq \clm$ is a
submodule of $H^2(\D^2)$, by Lemma \ref{Lemma}, it follows that
\begin{equation}\label{eq6}
\text{rank} (P_{\cle} M_z|_{\cle}, P_{\cle} M_w|_{\cle}) \leq
\text{rank} (M_z|_{\cls}, M_w|_{\cls}) = \text{rank} (\cls).
\end{equation}
Note that
\[
P_{\mathcal{E}} = (P_{\s_{\varphi}} \otimes P_{\q_{{\psi}}}) \oplus
(P_{\q_{\varphi}}\otimes P_{\s_{{\psi}}}).
\]
and hence, an easy calculation yields
\[
P_{\mathcal{E}}M_{z}|_{\mathcal{E}} =
(M_{z}|_{\s_{\varphi}} \otimes P_{\q_{{\psi}}})
\oplus (P_{\q_{\varphi}}M_{z}|_{\q_{\varphi}}\otimes
P_{\s_{{\psi}}}),
\]
and
\[
P_{\mathcal{E}}M_{w}|_{\mathcal{E}} = (P_{\s_{\varphi}} \otimes
P_{\q_{{\psi}}}M_{w}|_{\q_{{\psi}}}) \oplus (P_{\q_{\varphi}}\otimes
M_{w}|_{\s_{{\psi}}}).
\]
Therefore it follows from the above equalities that
$(\s_{\varphi^2}\otimes \q_{\psi}) \oplus ( \q_{\varphi}\otimes
\s_{{\psi}^2})$ is a joint $(P_{\mathcal{E}}M_{z}|_{\mathcal{E}},
P_{\mathcal{E}}M_{w}|_{\mathcal{E}})$ invariant subspace of $\cle$.
Set
\[
\tilde{\mathcal{E}} = \mathcal{E} \ominus ((\s_{\varphi^2}\otimes
\q_{\psi}) \oplus ( \q_{\varphi}\otimes \s_{{\psi}^2})).
\]
Notice that for any inner function $\theta \in H^\infty(\D)$, we
have
\[
\cls_\theta \ominus \cls_{\theta^2} = \theta \clq_{\theta}.
\]
From this and the representation of $\cle = ( \s_{\varphi} \otimes
\q_{{\psi}}) \oplus (\q_{\varphi} \otimes \s_{{\psi}})$ it follows
that
\[
\begin{split}
\tilde{\mathcal{E}} & = (( \s_{\varphi} \otimes \q_{{\psi}}) \oplus
(\q_{\varphi} \otimes \s_{{\psi}}))
\ominus ((\s_{\varphi^2}\otimes \q_{\psi}) \oplus (\q_{\varphi}\otimes \s_{{\psi}^2}))\\
& \hspace{0.1in} = (\varphi\q_{\varphi} \otimes \q_{{\psi}}) \oplus
(\q_{\varphi}\otimes {\psi}\q_{{\psi}}).
\end{split}
\]
Then Lemma \ref{Lemma} and (\ref{eq6}) implies that
\[
\text{rank} (P_{\tilde{\cle}} M_z|_{\tilde{\cle}}, P_{\tilde{\cle}}
M_w|_{\tilde{\cle}}) \leq \text{rank} (P_{\cle} M_z|_{\cle},
P_{\cle} M_w|_{\cle}) \leq \text{rank} (\cls) \leq 2.
\]
To finish the proof of the theorem it is now enough to prove the
following:
\[
\text{rank} (P_{\tilde{\cle}} M_z|_{\tilde{\cle}}, P_{\tilde{\cle}}
M_w|_{\tilde{\cle}})>1.
\]
Equivalently, it is enough to prove that the set $\{\xi\}$, for any
$\xi \in \tilde{\cle}$, is not a generating set corresponding to
$(P_{\tilde{\mathcal{E}}}M_{z}|_{\tilde{\mathcal{E}}},
P_{\tilde{\mathcal{E}}}M_{w}|_{\tilde{\mathcal{E}}})$. Equivalently,
given $\xi \in \tilde{\cle}$, we show that there exists $\eta_{\xi} (\neq 0) \in
\tilde{\cle}$ such that
\[
\langle (z^p \otimes w^q) \xi, \eta_{\xi}\rangle = 0 \quad \quad
(p,q\in \mathbb{N}).
\]
To this end, let $\{f_i\}$ and $\{g_j\}$ be orthonormal bases of
$\q_{\varphi}$ and $\q_{\psi}$, respectively, and let $\xi \in
\tilde{\cle}$ where
\[
\xi = (\sum_{k, l} a_{kl} {\varphi} f_k\otimes g_l) \oplus
(\sum_{k,l} b_{kl} f_k\otimes \psi g_l),
\]
$\{a_{kl}\}, \{b_{kl}\} \subseteq \mathbb{C}$, and
\[
\sum_{k, l} |a_{kl}|^2, \sum_{k,l} |b_{kl}|^2 <\infty.
\]
Again we observe that for any inner function $\theta \in
H^\infty(\D)$ and $f = \sum_{m \geq 0} c_m z^m \in \clq_{\theta}$ we
have
\[
M_z^* ( \theta \bar{f}) \in \clq_{\theta},
\]
where $\bar{f} = \sum_{m \geq 0} \bar{c}_m e^{- i m t} \in
L^2(\mathbb{T})$. This follows from the fact that $\theta$ is a
bounded holomorphic function on $\D$ and $M_z^* (\theta \bar{f}) \perp
z^m$ for all $m < 0$ (which gives that $M_z^* (\theta \bar{f}) \in
H^2(\D)$), and then $M_z^* (\theta \bar{f}) \perp \theta z^m$ in
$L^2(\mathbb{T})$ for all $m \geq 0$ (which gives that $M_z^* (\theta
\bar{f}) \in \clq_{\theta}$). It should be noted that
$M_z^* (\theta \bar{f}) =\theta \overline{zf} = C_{\theta}(f)$, where the conjugation map
$C_{\theta}:\q_{\theta}\to \q_{\theta}$, $f\mapsto M_z^*(\theta \bar{f})$,
is called a $C$-symmetry and it is used extensively in the
study of Toeplitz operators on model spaces (for more details see \cite{Gar}).

\noindent Coming back to our context, this immediately yields that
\[
M_z^* ({\varphi} \overline{f_k}) \otimes M_w^* ({\psi}\overline {g_l})
\in \q_{\varphi} \otimes \q_{\psi} \quad \quad (k, l\geq 0),
\]
and hence $s_0\otimes s_1, t_0\otimes t_1 \in \q_{\varphi} \otimes
\q_{\psi}$, where
\[
s_0\otimes s_1 := - \sum_{k,l} \overline{a}_{kl} M_z^*({\varphi}
\overline{f}_k) \otimes M_w^* ({\psi}\overline{g}_l) = - (M_z^*
\otimes M_w^*)({\varphi} \otimes{\psi}) (\sum_{k, l} \bar{a}_{kl}
\bar{f}_k \otimes \bar{g}_l)
\]
and
\[
t_0\otimes t_1 := \sum_{k,l} \overline{b}_{kl} M_z^*({\varphi}
\overline{f}_k) \otimes M_w^* ({\psi}\overline{g}_l) = (M_z^*
\otimes M_w^*)({\varphi} \otimes{\psi}) (\sum_{k, l} \bar{b}_{kl}
\bar{f}_k \otimes \bar{g}_l).
\]
Set
\[
\eta_{\xi} = ({\varphi} t_0\otimes t_1) \oplus (s_0 \otimes {\psi}
s_1) \in \tilde{\cle}.
\]
Then $\eta_{\xi}\neq 0$ and for every $p, q\in \mathbb{N}$ we have
\[
\begin{split}
\langle (z^p\otimes {w}^q) \xi, \eta_{\xi}\rangle & = \langle
(z^p\otimes w^q) ((\sum_{k, l} a_{k l} {\varphi} f_k\otimes g_l)
\oplus (\sum_{k,l} b_{kl} f_k\otimes \psi g_l)), (\varphi t_0\otimes
t_1) \oplus (s_0 \otimes \psi s_1)\rangle
\\
& = \langle (z^p\otimes w^q) (\sum_{k, l} a_{k l} \varphi f_k\otimes
g_l), \varphi t_0\otimes t_1 \rangle
\\
&\hspace{1in} + \langle (z^p\otimes w^q) (\sum_{k,l} b_{kl} f_k
\otimes \psi g_l), s_0 \otimes \psi s_1\rangle
\\
& = \langle (z^p\otimes w^q) (\sum_{k, l} a_{k l} f_k \otimes g_l),
t_0\otimes t_1 \rangle + \langle (z^p\otimes w^q)(\sum_{k,l} b_{kl}
f_k\otimes g_l), s_0 \otimes s_1 \rangle
\\
& = \langle (z^{p+1}\otimes w^{q+1}) (\sum_{k, l} a_{k l}~
f_k\otimes g_l),
({\varphi}\otimes{\psi})(\sum_{k,l=1}^{\infty} \bar{b}_{kl} \bar{f}_k\otimes \bar{g}_l)\rangle
\\
&\hspace{1in} - \langle (z^{p+1}\otimes w^{q+1})(\sum_{k,l} b_{kl}
f_k\otimes g_l), ({\varphi} \otimes{\psi}) (\sum_{k, l} \bar{a}_{kl}
\bar{f}_k \otimes \bar{g}_l)\rangle
\\
&=0 .
\end{split}
\]
We have thus shown that $\{\xi\}$ is not a minimal generating
subset of $\tilde{\cle}$ with respect to
$(P_{\tilde{\mathcal{E}}}M_{z}|_{\tilde{\mathcal{E}}},
P_{\tilde{\mathcal{E}}}M_{w}|_{\tilde{\mathcal{E}}})$ as desired.
\end{proof}

As a consequence of the above theorem we have the following
corollary which provides an affirmative answer of the question
raised by Douglas and Yang \cite{DY}.

\begin{Corollary}
Let $\cls$ be a co-doubly commuting submodule of $H^2(\D^2)$. Then
rank $(\cls)=1$ if and only if $\cls = \Theta H^2(\D^2)$ for some
one variable inner function $\Theta \in H^\infty(\D)$.
\end{Corollary}
\begin{proof}
If $\cls = \Theta H^2(\D^2)$ for some one variable inner function
$\Theta \in H^\infty(\D)$, then $\cls \cong H^2(\D^2)$ and hence
rank $\cls =1$. To prove the the sufficient part let $\cls$ be a
rank one co-doubly commuting submodule of $H^2(\D^2)$. Then there
exist quotient modules $\clq_1$ and $\clq_2$ of $H^2(\D)$ such that
(see \cite{bshift, sarkar})
\[
\cls = (\q_1 \otimes \q_2)^{\perp}.
\]
Since rank $(\cls)=1$, it follows from Theorem \ref{mainthm} that
$\clq_j = H^2(\D)$, for some $j = 1, 2$. This shows that
\[
\cls = \s_{\varphi}\otimes H^2(\D), \quad \mbox{or} \quad \quad \cls
= H^2(\D)\otimes \s_{\psi},
\]
for some inner functions $\varphi, \psi \in H^\infty(\D)$. This
concludes the proof of the corollary.
\end{proof}

There is now the following interesting and natural question: Let $m
\geq 2$ and let $\{\vp_j\}_{j=1}^m \subseteq H^\infty(\D)$ be inner
functions. Is then
\[\mbox{rank~} (\clq_{\vp_1} \otimes \cdots \otimes
\clq_{\vp_m})^\perp = m?
\]
Our present approach does not seem to work for $m> 2$ case.

\vspace{0.1in} \noindent\textbf{Acknowledgement:}
The first named author acknowledge Fulbright-Nehru Postdoctoral Research Fellowship (Award No. 2164/FNPDR/2016)
and University of New Mexico for warm hospitality.
The second author's research work is supported by DST-INSPIRE Faculty
Fellowship No. DST/INSPIRE/04/2015/001094. The research of the third
author is supported in part by NBHM (National Board of Higher
Mathematics, India) Research Grant NBHM/R.P.64/2014.

\end{document}